\documentclass[reqno,11pt]{amsart}

\usepackage[utf8]{inputenc}
\usepackage{amsthm}
\usepackage{amssymb}
\usepackage{amsmath}
\usepackage{mathtools}
\usepackage{pdfpages}
\usepackage{stmaryrd}
\usepackage{float}
\usepackage{tikz-cd}
\usepackage{setspace}
\usepackage[all]{xy}
\usepackage{enumitem}

\usepackage{microtype}

\definecolor{darkgreen}{RGB}{0,128,0}
\definecolor{darkblue}{RGB}{0,0,205}
\definecolor{myred}{RGB}{227,0,34}
\definecolor{chocolate}{RGB}{123,63,0}
\usepackage[colorlinks=true,pdftex,unicode=true,linktocpage,bookmarksopen,hypertexnames=false,citecolor=darkgreen, linkcolor=darkblue, urlcolor=myred]{hyperref}

\usepackage{tikz-cd}
\usepackage{caption}

\def\ZZ{ \mathbb{Z} }
\def\NN{ \mathbb{N} }
\DeclareMathOperator{\Spec}{Spec}
\renewcommand{\max}{\mathrm{max}}

\newcommand{\x}{\mathbf{x}}
\newcommand{\y}{\mathbf{y}}
\newcommand{\UU}{\mathcal{U}}
\newcommand{\A}{\mathcal{A}}

\newcommand{\defit}[1]{{\textsf{#1}}}

\newcommand{\vprod}{\sideset{}{\mkern -2mu\raisebox{0.2em}{$\vphantom{\prod}^v$}} \prod}

\makeatletter
\numberwithin{equation}{section}

\theoremstyle{plain}
\newtheorem{lemma}{Lemma}[section]
\newtheorem{theorem}[lemma]{Theorem}
\newtheorem{proposition}[lemma]{Proposition}
\newtheorem{corollary}[lemma]{Corollary}

\theoremstyle{definition}
\newtheorem{definition}[lemma]{Definition}

\newtheorem{examples}[lemma]{Examples}

\theoremstyle{remark}
\newtheorem{remark}[lemma]{Remark}

\renewcommand{\epsilon}{\varepsilon}
\renewcommand{\theta}{\vartheta}
\renewcommand{\phi}{\varphi}

\setlist[enumerate,1]{itemsep=0.05cm, label=\textup{(}\arabic*\textup{)}}
\setlist[enumerate,2]{label=\textup{(}\roman*\textup{)}}

\setlist[itemize,1]{itemsep=0.05cm}

\makeatletter
\@namedef{subjclassname@2020}{%
	\textup{2020} Mathematics Subject Classification}
\makeatother

\title{On Class Groups of Upper Cluster Algebras}
\author{Mara Pompili}
\address{University of Graz, Department of Mathematics and Scientific Computing, NAWI Graz, Heinrichstrasse 36, 8010 Graz, Austria}
\email{mara.pompili@uni-graz.at}
\thanks{The author was partially supported by the Austrian Science Fund (FWF), projects W1230 and P36742-N}

\begin{document}
	
	\maketitle
	\begin{abstract}  We compute the class group of a full rank upper cluster algebra in terms of its exchange polynomials. As a corollary, we recover a theorem by Cao, Keller, and Qin from 2023 characterizing the UFDs among these algebras. Furthermore, under the additional hypothesis of acyclicity, we obtain a result by Garcia Elsener, Lampe, and Smertnig from 2019.  Moreover, we show that all cluster and upper cluster algebras are finite factorization domains, meaning that every non-unit factors as a product of atoms (or irreducibles) and, for each element, there are only finitely many such factorizations up to order and associates. This strengthens another result by Cao, Keller, and Qin showing that cluster and upper cluster algebras are atomic.  \end{abstract}
	
	\section*{Introduction}
	In recent years, cluster algebras have emerged as a powerful tool in various branches of mathematics, including algebraic geometry, representation theory, and mathematical physics. These structures provide deep insights into the interplay between algebraic and combinatorial phenomena. Strongly related to a cluster algebra is its upper cluster algebra. Upper cluster algebras are overrings of cluster algebras that are in general better behaved from a ring-theoretic perspective and often even coincide with the cluster algebra, for instance, in the locally acyclic case  \cite{BFZ05, M14}.
 	
	Cluster algebras were introduced by Fomin and Zelevinsky \cite{FZ02} in the early 2000s, with the intent of investigating dual canonical bases and total positivity within semisimple Lie groups. Quickly, cluster algebra theory has evolved into a distinct and self-sustained domain, transcending its initial motivations, cf. for example the surveys \cite{FZ03, K12, L10, LW14, N22, Z05}. 
	It has found numerous applications, particularly in fields like representation theory, Poisson geometry, and Teichm\"uller theory. 
    In this paper we focus on ring-theoretic properties of (upper) cluster algebras, following in the footsteps of \cite{CKQ22, GELS19, GLS13, M13}.

    The Laurent phenomenon (Theorem \ref{thm:laurentphenomenom}) implies that every seed gives rise to an embedding of the cluster algebra into a Laurent polynomial ring. The upper cluster algebra is defined as the intersection of all these Laurent polynomial rings. Therefore upper cluster algebras are an upper bound for cluster algebras. The intrinsic characteristics of upper cluster algebras simplify the exploration of their factorization properties. An even more accessible scenario emerges within the case of full rank cluster algebras, where upper cluster algebras take the form of a finite intersection of Laurent polynomial rings.
 
	Locally acyclic cluster algebras and full rank upper cluster algebras are Krull domains. 
    Krull domains \cite{FOSSUM} are one of the central objects in the study of non-unique factorizations. They form an important generalization of unique factorization domains (UFDs), because they possess unique factorization on the level of divisorial ideals, which allows one to systematically study factorizations of elements.
	A key invariant of Krull domains are their class groups, which measure the failure of unique factorization in these domains. In particular, a domain is a UFD if and only if it is a Krull domain with trivial class group. In the present paper, we study class groups of upper cluster algebras that satisfy the {\em starfish condition at one seed} (see Definition \ref{def:starfish}), in particular this includes all full rank upper cluster algebras. 
	
	The exploration of factorization properties within cluster algebras was initially undertaken by Geiss, Leclerc, and Schröer \cite{GLS13}. Subsequently, Garcia Elsener, Lampe, and Smertnig achieved a breakthrough by computing class groups and their ranks for acyclic cluster algebras \cite{GELS19}. 
	Very recently, Cao, Keller, and Qin \cite{CKQ22} switched the focus to upper cluster algebras. 
  They provided a complete characterization of full rank upper cluster algebras that are UFDs and they also exhibited local factorization properties using {\em valuation pairs.}
	
	In this paper, we simultaneously generalize results of Garcia Elsener, Lampe, and Smertnig and Cao, Keller, and Qin, computing the class group of full rank upper cluster algebras in terms of the exchange polynomials. We show that the class group $\mathcal{C}(\UU)$ of a full rank upper cluster algebra $\mathcal{U}$ is a finitely generated free abelian group with rank $r=t-n$, where $t$ is the number of irreducible factors of their exchange polynomials and $n$ is the number of exchangeable variables (see Theorem \ref{thm:rankclassgroup}). Furthermore, we show that, in this case, the upper cluster algebra $\UU$ contains infinitely many height-1 prime divisors in each class. This leads to a fascinating dichotomy between those upper cluster algebras that are UFDs and those that are not: in the latter case every finite set $L\subseteq \ZZ_{\ge 2}$ can be realized as a length set of some element. This result even applies to a larger class of upper cluster algebras than the full rank ones (see Section \ref{sec3}).

    We remark that, compared to the results of Garcia Elsener, Lampe, and Smertnig \cite{GELS19}, we need the additional assumption on full rankness; on the other hand, we do not need the acyclicity hypothesis present in their work.
    Thus, while our results do not fully generalize results of \cite{GELS19} in every case, they do so in a large and important subclass.

    Whereas factorization properties in cluster algebras that are Krull domains have been studied, so far it has not been considered which factorization properties hold in full generality in arbitrary (upper) cluster algebras. 
    For instance, the Markov cluster algebra is not a Krull domain \cite[Section 6]{GELS19}.
    Interestingly, its upper cluster algebra is however a UFD.
    Indeed, the question of whether upper cluster algebras that are not Krull domains exist remains an open problem.
    In the setting of (upper) cluster algebras without any extra conditions, we establish that every (upper) cluster algebra is a finite factorization domain (see Theorem~\ref{prop:FF}). That is, every element has a factorization into atoms (i.e., irreducible elements) and for a given element there are only finitely many such factorizations up to permutation and associativity.
    This improves on a recent result by which such algebras are bounded factorization domains (i.e., every element has at most finitely many different factorization lengths) \cite[Appendix A]{CKQ22}.
    
	The paper is organized as follows. In Section \ref{sec1}, we recall basic definitions and results on cluster algebras, factorization theory, and Krull domains. In Section \ref{sec2}, we study some factorization properties of (upper) cluster algebras, proving that cluster algebras and upper cluster algebras are finite factorization domains. The proof we will not necessitate preliminaries on Krull domains. In Section \ref{sec3}, we focus on full rank upper cluster algebras, computing their class groups. Finally in Section \ref{sec4}, we give an interpretation in terms of multiplicative ideal theory of the notion of valuation pairing and local factorization introduced in \cite{CKQ22}.
	
	
	\subsection*{Notations and assumptions}Throughout the paper we consider cluster algebras of geometric type, allowing frozen variables. However, we always assume that all frozen variables are invertible. Moreover, $K$ will denote a field of characteristic zero, or the ring $\mathbb{Z}$ of integers.
	If the base ring $K$ is a field, we assume that the underlying ice quiver $\Gamma(B)$ of the exchange matrix $B$ of our cluster algebra has no isolated exchangeable vertices. 
	A domain is a non-zero commutative ring $A$ without non-zero zero-divisors. We denote by $A^\times$ its group of units, by $A^\bullet=A\setminus \{0\}$ its monoid of non-zero elements, and by $\mathbf{q}(A)$ its quotient field.
	We denote by $\NN$ the semigroup of positive integers and by $\NN_0$ the monoid $\NN\cup \{0\}$. Moreover, if $n\in \NN$, we denote by $[1,n]$ the set $\{1,\dots,n\}$.

	\section{Preliminaries}\label{sec1}

	\subsection{Quivers}
	A  \defit{quiver} is a finite directed graph. Thus, it is a tuple $\mathcal{Q}=(\mathcal{Q}_0,\mathcal{Q}_1,s,t)$ where $\mathcal{Q}_0$ (the set of vertices) and $\mathcal{Q}_1$ (the set of arrows) are finite sets and $s,t\colon \mathcal{Q}_1\to \mathcal{Q}_0$ are maps (the sources and the targets). 
	We write $\alpha\colon i \to j$ to indicate that $\alpha$ is an arrow in $\mathcal{Q}_1$ with $s(\alpha)=i$ and $t(\alpha)=j.$
	
	\begin{center}	
	{\em Through the entire paper, a quiver has no oriented cycles of length one or two.}
	\end{center}

    \begin{definition}[Ice quivers]
		An \defit{ice quiver} is a quiver $\mathcal{Q}=(\mathcal{Q}_0,\mathcal{Q}_1,s,t)$ together with a partition of $\mathcal{Q}_0$ into \defit{exchangeable} and \defit{frozen} vertices, with the assumption that there are no arrows between two frozen vertices.
		The \defit{exchangeable part} of $\mathcal{Q}$ is the subquiver on the set of exchangeable vertices. We say that $\mathcal{Q}$ is \defit{acyclic} if its exchangeable part is an acyclic quiver, i.e., if it does not contain any oriented cycles.
	\end{definition}
	
	Let $n,m\in \NN_0$ such that $n+m>0$. Let $\mathcal{Q}=(\mathcal{Q}_0,\mathcal{Q}_1,s,t)$ be an ice quiver with exchangeable vertices $[1,n]$ and frozen vertices $[n+1,n+m].$ We can associate to $\mathcal{Q}$ a matrix $B=B(\mathcal{Q})=(b_{ij})\in \mathcal{M}_{(n+m)\times n}(\ZZ)$ defined by \[b_{ij}=|\,\alpha\colon i \to j\,|-|\,\alpha\colon j \to i\,|.\]
	
	\smallskip
	
	Given an $(n+m)\times n$ matrix, its \defit{principal part} is the submatrix supported on the first $n$ rows. Notice that the principal part of $B(\mathcal{Q})$ is skew-symmetric.
	
	\begin{definition}[Exchange matrices]
		A matrix $B=(b_{ij})\in \mathcal{M}_{n\times n}(\ZZ)$ is \defit{skew-symmetrizable} if there exists a diagonal matrix $D\in \mathcal{M}_{n\times n}(\NN)$ such that $DB$ is skew-symmetric. 
		An $(n+m)\times n$ integer matrix is an \defit{exchange matrix} if its principal part is skew-symmetrizable. 
	\end{definition}
	For a matrix $B=(b_{ij})\in \mathcal{M}_{n\times n}(\ZZ)$ being skew-symmetrizable is equivalent to the existence of positive integers $d_1,\dots,d_n$ such that $d_ib_{ij}=-d_jb_{ji}$ for every $i,j\in[1,n],$ hence if $B$ is a skew-symmetrizable matrix, either $b_{ij}=b_{ji}=0$ or $b_{ij}b_{ji}< 0$, so in particular $b_{ii}=0$.
	
	\begin{remark}
		Let $\mathcal{Q}$ be an ice quiver, then the matrix $B(\mathcal{Q})$ is an exchange matrix with skew-symmetric principal part.
		Conversely, we can associate to any exchange matrix $B$ an ice quiver $\Gamma(B)$. The exchangeable and frozen vertices of $\Gamma(B)$ are $[1,n]$ and $[n+1,n+m],$ respectively. The arrows of $\Gamma(B)$ are defined as follows: if $b_{ij}>0$, add $b_{ij}$ arrows from $i$ to $j$. If $i$ is a frozen vertex and $b_{ij}<0$, then add also $-b_{ij}$ arrows from $j$ to $i.$
		If $B$ has skew-symmetric principal part, then $B(\Gamma(B))=B.$ 
	\end{remark}   
	
	\subsection{Seeds, mutations and cluster algebras} 
	A \defit{cluster} is a pair $(\x,\y)$ with $\x=(x_1$,$\ldots$,~$x_n)$ and      
	\mbox{$\y=(x_{n+1}$,$\ldots$,~$x_{n+m})$} such that $(x_1$,$\ldots$,~$x_n$,$\ldots$,~$x_{n+m})$ are $n+m$ algebraically independent indeterminates over $K$. We refer to the elements of $\x$ as \defit{exchangeable variables} and to the elements of $\y$ as \defit{frozen variables}.	
	Given a cluster, the field $\mathcal{F}=\mathbf{q}(K)(x_1$,$\ldots$,~$x_{n+m})$ is called the \defit{ambient field}.
	
	\begin{definition}[Seeds]\label{def:seed}
		A \defit{seed} is a triple $\Sigma=(\x,\y, B)$ such that $(\x,\y)$ is a cluster and $B$ is a $(n+m)\times n$  exchange matrix.
        We always tacitly assume $\x=(x_1,\ldots,x_n)$ and $\y=(x_{n+1},\ldots,x_{n+m})$.
		A seed is called \defit{acyclic} if the ice quiver $\Gamma(B)$ is acyclic.
	\end{definition}
	
	We identify two seeds $\Sigma=(\x,\y,B)$ and $\Sigma'=(\x',\y',B')$ if there exists a permutation $\sigma\in S_{n+m}$ such that $\sigma(i)\in [1,n]$ for all $i\in [1,n]$ and 
    \begin{itemize}
        \item $b_{ij}=b'_{\sigma(i),\sigma(j)}$ for every $i,j\in [1,n+m]$;
    
        \item $x_i=x'_{\sigma(i)},\, y_j=y'_{\sigma(j)}$ for every $i\in[1,n]$ and $j\in [n+1,n+m]$.
    \end{itemize}
	
	\begin{definition}[Mutation of seeds]
		Let $\Sigma=(\x,\y, B)$ be a seed with ambient field $\mathcal{F}$. Fix an exchangeable index $i\in [1,n].$ The \defit{mutation} of $\Sigma$ in direction $i$ is the triple $\mu_i(\Sigma)=(\x_i,\y_i,B_i)$ defined as follows
		\begin{itemize}
			\item[(a)] $\x_i=(x_1,\dots,x_i',\dots,x_n)$ with $$x_i'=\frac{1}{x_i}\left(\prod_{b_{ki}>0}x_k^{b_{ki}}+\prod_{b_{ki}<0}x_k^{-b_{ki}}\right)\in \mathcal{F};$$
			\vspace{0.05cm}
			\item[(b)] $\y_i=\y$;
			\vspace{0.05cm}
			\item[(c)] $B_i=(b'_{jk})\in  \mathcal{M}_{(n+m)\times n}(\ZZ)$ with \[b_{jk}'=\begin{cases}-b_{jk}& \text{if $j=i$ or $k=i$;}\\ b_{jk}+\frac{1}{2}(|b_{ji}|b_{ik}+b_{ji}|b_{ik}|) & \text{otherwise.} \end{cases} \]
		\end{itemize}
	\end{definition}
	
	\smallskip
	
	One can prove that $\mu_i(\Sigma)$ is a seed with the same number of exchangeable and frozen variables and the same ambient field as $\Sigma$, and $B_i$ has the same rank as $B$. Notice that $(\mu_i\circ \mu_i)(\Sigma)=\Sigma.$

	\begin{definition}[Exchange polynomials]
		Let $\Sigma=(\x,\y, B)$ be a seed. Suppose that $i\in[1,n]$ is an exchangeable index. The polynomial $$f_i:=x_ix_i'\in K[\x,\y]$$ is called the \defit{exchange polynomial} associated to $x_i$ (with respect to the seed $\Sigma$). 
	\end{definition}
	
	Mutations induce an equivalence relation on seeds. Indeed we say that two seeds $\Sigma=(\x,\y,B)$ and $\Sigma'=(\mathbf{z},\y,C)$ are \defit{mutation-equivalent} if there exist $i_1,\dots,i_k\in [1,n]$ such that $\Sigma'=(\mu_{i_1}\circ\cdots\circ \mu_{i_k})(\Sigma)$. In this case, we write $\Sigma\sim \Sigma',$ or, if no confusion can arise, $\x\sim\mathbf{z}$. Denote by $\mathcal{M}(\Sigma)$ the mutation equivalence class of $\Sigma$ and by $\mathcal{X}=\mathcal{X}(\Sigma)$ the set of all exchangeable variables appearing in $\mathcal{M}(\Sigma).$
	
	\begin{definition}[Cluster algebras]
		Let $\Sigma=(\x,\y, B)$ be a seed. The \defit{cluster algebra} associated to $\Sigma$ is the $K$-algebra $$\A=\A(\Sigma)=K[x,y^{\pm 1}\mid x\in \mathcal{X}, y\in \y].$$ The elements $x\in \mathcal{X}$ are called \defit{cluster variables} of $\A(\Sigma)$; the cluster variables in the initial seed $\Sigma$ are called \defit{initial cluster variables}; the elements $y \in \y$ are called \defit{frozen variables.}
	\end{definition}
	
	In the literature one can also find the definition of a cluster algebra associated to a seed $\Sigma=(\x,\y, B)$ as the $K$-algebra $K[x,y\mid x\in \mathcal{X}, y\in \y]$. However we only deal with the case of invertible frozen variables.
	
	For any seed $\Sigma=(\x,\y, B)$, denote by $$\mathcal{L}_{\x}=K[u^{\pm 1}\mid u\in \x\cup \y]$$ the localization of $K[u\mid u\in \x\cup \y]$ at $S_\x:=\{\,{x_1^{a_1}\cdots x_{n+m}^{a_{n+m}}\mid a_i\in \NN_0}\,\} $ and by $$\mathcal{L}_{\x,\ZZ}=\ZZ[u^{\pm 1}\mid u\in \x\cup \y]$$ the localization of $\ZZ[u\mid u\in \x\cup \y]$ at $S_\x. $
	
	\begin{definition}[Upper cluster algebras]
		Let $\Sigma=(\x,\y, B)$ be a seed. The \defit{upper cluster algebra} associated to $\Sigma$ is the $K$-algebra $$\UU=\UU(\Sigma)=\bigcap_{\mathbf{z} \sim \x} \mathcal{L}_{\mathbf{z}}.$$
	\end{definition}
	
	An upper cluster algebra $\UU(\Sigma)$ is called a \defit{full rank upper cluster algebra} if its initial exchange matrix has full rank. For a full rank upper cluster algebra $\UU$, every exchange matrix of $\UU$ has full rank, since, as mentioned above, mutations preserve the rank (\cite[Lemma 3.2]{BFZ05}). 
	
	The following result, known as {\em Laurent phenomenon}, is due to Fomin and  Zelevinsky and explains the relation between cluster algebras and upper cluster algebras.
	\begin{theorem}[Laurent phenomenon, \cite{FZ02}]\label{thm:laurentphenomenom}
		Let $\Sigma=(\x,\y, B)$ be a seed. Let $\mathcal{X}(\Sigma)$ be the set of cluster variables associated to $\Sigma$ and let $\A(\Sigma),\,\UU(\Sigma)$ be the cluster algebra and the upper cluster algebra associated to $\Sigma,$ respectively. 
		Then  $$\mathcal{X}(\Sigma)\subseteq \bigcap_{\mathbf{z}\sim \x} \mathcal{L}_{\mathbf{z},\ZZ}\,,\qquad \text{in particular,}\qquad \A(\Sigma)\subseteq \UU(\Sigma). $$
	\end{theorem}
	
	Given a seed $\Sigma=(\x,\y,B)$, the following inclusions hold \begin{equation}\label{eq:inclusions}\A(\Sigma)\subseteq \UU(\Sigma)\subseteq \bigcap_{i=0}^{n}\mathcal{L}_{\x_i}.\end{equation} In certain instances, the inclusions \eqref{eq:inclusions} turn into equalities, see Examples \ref{ex:1}.
	
	Locally acyclic cluster algebras were introduced by Muller in \cite{M13} to generalize acyclic cluster algebras maintaining some of their properties. They are a much large class than the acyclic cluster algebras, in particular they include cluster algebras arising from marked surfaces. Locally acyclic cluster algebras are finitely generated, noetherian, and integrally closed (\cite[Theorem 4.2]{M13}).  Moreover, the following theorem holds.
	
	\begin{theorem}[{\cite[Theorem 2]{M14}}]
		If $\A$ is locally acyclic, then $\A=\UU$.
	\end{theorem}
	Another class for which one of the inclusions becomes an equality is the one of full rank upper cluster algebra. The following result is known as Starfish lemma.
	\begin{theorem}[{\cite[Corollary 1.9]{BFZ05}}]\label{thm:starfish}
		Let $\UU$ be a full rank upper cluster algebra. Then  $$\UU=\bigcap_{i=0}^{n}\mathcal{L}_{\x_i},$$ for all seeds $(\x,\y,B)$.
	\end{theorem}
	
	Assume that $K$ is a field and that $\Sigma$ is an isolated seed, i.e. $\Sigma=\left(\left(x_1,\dots,x_n\right),\emptyset,0\right)$. Then all the inclusions \eqref{eq:inclusions} are equalities. Indeed $$\A(\Sigma)=\UU(\Sigma)=\bigcap_{i=0}^n\mathcal{L}_{\x_i}=K[x_1^{\pm 1},\dots,x_{n}^{\pm 1}].$$  
	\begin{examples}\label{ex:1}
	    \begin{enumerate}
		\item The first example of a cluster algebra for which $\A\ne\UU$ is due to Berenstein, Fomin, and Zelevinsky \cite{BFZ05} and it is the {\em Markov cluster algebra}. It is the cluster algebra with base ring $K=\mathbb Z$ associated to the following quiver $\mathcal{Q}$
		
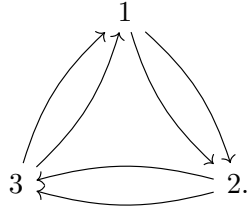
        \[
        \begin{tikzcd}[row sep=50pt]
            & 1 \ar[dr,bend left=15] \ar[dr, bend right=15] & \\
            3 \ar[ur, bend left=15] \ar[ur, bend right=15] & & 2. \ar[ll, bend left=15] \ar[ll, bend right=15]
        \end{tikzcd}
        \vspace{8pt}
        \] 
        $\A(\mathcal{Q})$ is a $\NN$-graded algebra \cite{M13}, with the degree of all cluster variables equal to $1$. For any cluster $(x_1,x_2,x_3)$,  the element $M=\frac{x_1^2+x_2^2+x_3^2}{x_1x_2x_3}$ is in $\UU(\mathcal{Q})$, but it has graded degree $-1$, so it is not in $\A(\mathcal{Q})$. Moreover, the upper cluster algebra is factorial and it is given by $\ZZ[x_1,x_2,x_3,M]$ (\cite[Proposition 6.7]{M15}).
		
		\smallskip
		
		\item Consider the quiver $A_3=\begin{tikzcd}[column sep=small]
		1 \arrow[r] & 2\arrow[r] & 3
		\end{tikzcd}$ and the element $$s=\frac{1+x_2}{x_1x_3}\in K(x_1,x_2,x_3).$$ It is easy to see that $$s=\frac{x_1'}{x_3}=\frac{x_3'}{x_1}\in \bigcap_{i=0}^3\mathcal{L}_{\x_i}. $$ However,  $(x_1',x_2,x_3')$ is a seed, obtained by mutating $\x$ in direction 3 and then 1, and $s\notin K[{x_1'}^{\pm 1},x_2^{\pm 1},{x_3'}^{\pm 1}].$ Indeed, if we express $s$ in terms of the new seed we obtain $$s=\frac{x_1'x_3'}{1+x_2}$$ which is not a Laurent polynomial in $(x_1',x_2,x_3')$.
		Hence $\UU(A_3)\subsetneq  \bigcap_{i=0}^3\mathcal{L}_{\x_i}$.
	\end{enumerate}
	\end{examples}

	However, there are non-full rank upper cluster algebras for which the Starfish lemma holds. For example, any factorial upper cluster algebra satisfies the Starfish lemma as shown in \cite[Corollary 1.5]{GLS13}. 
 
	\begin{definition}\label{def:starfish}  Let $\UU$ be an upper cluster algebra.
		We say that $\UU$ \defit{satisfies the starfish condition at the seed $(\x,\y\,B)$} if $$\UU=\bigcap_{i=0}^n\mathcal{L}_{\mathbf{x}_i}.$$ 
	\end{definition}
	
	If $\UU$ has full rank, then Theorem \ref{thm:starfish} implies that $\UU$ satisfies the starfish condition at all seeds. We introduced this definition because for our results we don't need the starfish condition at all seeds but just at one of them (see Section \ref{sec3}).
	
	\begin{remark}\label{rmk:isolated}
		Let $\Sigma=(\x,\y,B)$ be a seed. Suppose that $i\in[1,n]$ is \defit{isolated}, i.e., $x_ix_i'=2$. If $K$ is a field, the index $i$ is isolated if and only if $x_i\in K^\times$. Therefore, if we freeze $i$, we obtain an algebra isomorphic to the original one. Hence, from now on, if $K$ is a field we assume without restriction that $[1,n]$ has no isolated indices.
	\end{remark}
	
	\subsection{Factorization Theory}
	In this section, we recall some basic notions of the theory of factorization. For more details see for instance \cite{GH06}. We always assume that our monoids are commutative and cancellative.
		
	\begin{definition}Let $H$ be a monoid.
		\begin{enumerate}
			\item Let $a,b \in H$. We say that $a | b$ if there exists an element $c\in H$ such that $b=ca.$ 
			\item Two elements $a,b\in H$ are \defit{associated} if there exists a unit $\epsilon\in H^\times$ such that $a=\epsilon b.$ In this case we write $a\simeq_H b$.
			\item A non-unit $u\in H$ is an \defit{atom}, or an \defit{irreducible}, if $u=ab$ with $a,b\in H$ implies $a\in H^\times$ or $b\in H^\times$.
			\item An atom $u\in H$ is a \defit{strong atom}, or an \defit{absolutely irreducible}, if for all positive integers $n>1$ the only factorization (up to associates) of $u$ is  $u^n=u\cdots u.$  
			\item A non-unit $p\in H$ is a \defit{prime} if $p|ab$ with $a,b\in H$ implies $p|a$ or $p|b$.
		\end{enumerate}
	\end{definition}
	
	It is well known that
	$$\text{prime}\implies \text{strong atom}\implies \text{atom}.$$
	We will denote by $\A(H)$ the set of atoms of the monoid $H$.
	
	\begin{definition}\label{def:factorial} Let $H$ be a monoid.
		\begin{enumerate}
			\item $H$ is \defit{atomic} if every non-unit of $H$ can be written as a finite product of atoms.
			\item $H$ is a \defit{bounded factorization monoid}, in short \defit{BF-monoid}, if it is atomic and for every non-unit $a\in H$ there exist  $\lambda(a)\in \NN_0$ such that $a=x_1\cdots x_n$ for $x_1,\dots,x_n$ atoms implies $n\le \lambda(a).$
			\item $H$ is a \defit{finite factorization monoid}, in short \defit{FF-monoid}, if it is atomic and every non-unit factors into atoms in only finitely many ways up to order and associates.
			\item $H$ is \defit{factorial}, or \defit{UF-monoid}, if it is atomic and every non-unit of $A$ factors in a unique way up to order and associates.
		\end{enumerate}
	\end{definition}
	
	The connection between the notions in Definition \ref{def:factorial} is described by the following picture:
	$$ \text{factorial monoid}\implies \text{FF-monoid} \implies \text{BF-monoid} \implies \text{atomic}.$$
	Recall these different characterizations of BF-monoids and FF-monoids (see \cite[Proposition 1.3.3 and Proposition 1.5.5]{GH06}).
	
	\begin{itemize}
		\item $H$ is a BF-monoid if and only if there exists a {\rm length function}, i.e., a map $\lambda\colon H\to \NN_0$ such that $\lambda(a)<\lambda(b)$ whenever $a|b$ and $a\not\simeq_H b$.
		\item $H$ is an FF-domain if and only if every $a\in H$ has only finitely many non-associated divisors.
	\end{itemize}
	
	Let $A$ be a domain. We say that $A$ is an \defit{atomic} (resp. \defit{BF}, \defit{FF}, \defit{factorial}) domain if the monoid $A^\bullet$ is atomic (resp. BF, FF, factorial).
	
	\bigskip
	
	We conclude this subsection with two well-known properties of localizations of domains.
	Denote by $\Spec(A)$ the set of all prime ideals of $A$. Recall that a \defit{multiplicative closed subset} of $A$ is a subset $S\subseteq A$ such that $1\in S,\, S\cdot S\subseteq S$, and $0\notin S$. 
	\begin{proposition}\label{prop:loc.ideals}
		Let $A$ be a domain and $S\subseteq A$ a multiplicatively closed set of $A$. There is an inclusion-preserving bijection 
		\begin{align*}
		\{\, \mathfrak{p} \in \Spec (A) \mid \mathfrak{p} \cap S = \emptyset \,\} &\longleftrightarrow \Spec (S^{-1}A),\\
		\mathfrak{p} &\mapsto S^{-1}\mathfrak{p} = \mathfrak{p} (S^{-1}A),\\
		\mathfrak{q} \cap A &\mapsfrom \mathfrak{q}.
		\end{align*}
		In particular, the bijection preserves the height of every prime ideal that does not meet $S$.
	\end{proposition}
	
	\begin{proposition}\label{lemma:loc.irr}
		Let $A$ be a factorial domain, and $S$ a multiplicatively closed set of $A$. Denote by $T$ the set of all the prime elements of $A$ that divide an element of $S$, and by $M=\mathcal{A}(A)\setminus T.$ 
		Then $S^{-1}T\subseteq (S^{-1}A)^{\times}$ and  $S^{-1}M\subseteq \mathcal{A}(S^{-1}A)$.
	\end{proposition}

	\subsection{Krull Domains}\label{subsec:krull}
	In this section, we recall some basic notions about Krull domains. For more details, see \cite{FOSSUM,GH06}.

	Recall that a \defit{discrete valuation ring}, in short DVR, is a factorial domain $A$ with a unique prime element $p$ (up to associates), that is, every $a \in A$ has a unique factorization of the form $a=up^n$
	with $u\in A^\times$ and $n\in \NN_0$. Equivalently, a domain $A$ is a discrete valuation ring if there exists a discrete valuation $v\colon \mathbf{q}(A)\to \ZZ\cup \{\infty\}$ such that $A= \{\,x\in\mathbf{q}(A)\mid v(x)\ge 0\,\}.$ With this notation, every $a\in A$ can be uniquely written as $a=up^{v(a)}$, with $u\in A^\times$ (see \cite[Theorem 2.3.8]{GH06}).
	\begin{definition}[Krull domains]
		Let $A$ be a domain and let $\mathfrak{X}(A)$ be the set of all height-1 prime ideals of $A$. Then $A$ is a \defit{Krull domain} if 
		\begin{enumerate}
			\item $A_\mathfrak{p}$ is a discrete valuation ring for every $\mathfrak{p}\in \mathfrak{X}(A);$
			\item $A=\bigcap_{\mathfrak{p}\in \mathfrak{X}(A)}A_{\mathfrak{p}}$;
			\item every non-zero element $a\in A$ is contained in at most a finite number of height-1 prime ideals of $A$.
		\end{enumerate}
	\end{definition}
	
	Equivalently, $A$ is a Krull domain if and only if there exists a family \{$v_i\}_{i\in I}$ of discrete valuations on the quotient field $\mathbf{q}(A)$ such that, for all $x\in \mathbf{q}(A)\setminus\{0\}$, 
	\begin{enumerate}
		\item[(i)]  $ v_i(x)= 0$ for all but finitely many $i\in I;$
		\vspace{0.05cm}
		\item[(ii)]  $x\in A$ if and only $v_i(x)\ge 0$ for all $i\in I.$
	\end{enumerate}
	
	\medskip

	Let $A$ be a domain.  For any non-empty subsets $X$,~$Y\subset \mathbf{q}(A)$, we define $$(Y:X)=\{\,a\in \mathbf{q}(A)\mid aX\subset Y\,\},\quad X^{-1}=(A:X),\quad\text{and}\quad X_v=(X^{-1})^{-1} .$$ 
	
	A subset $\mathfrak{c}\subseteq A$ is called a \defit{divisorial ideal} of $A$ if $\mathfrak{c}_v=\mathfrak{c}$ and a subset $\mathfrak{c}\subseteq \mathbf{q}(A)$ is called a \defit{fractional divisorial ideal} of $A$ if there exists some $x\in A^\bullet$ such that $x\mathfrak{c}$ is a divisorial ideal of $A$. We denote by $\mathcal{I}_v(A)$ the set of all divisorial ideals of $A$ and by $\mathcal{F}_v(A)$ the set of all fractional divisorial ideal of $A$. For fractional ideals $\mathfrak{a}$,~$\mathfrak{b}\in \mathcal{F}_v(A)$, define their $v$-product by $\mathfrak{a}\cdot_v \mathfrak{b}=(\mathfrak{a}\mathfrak{b})_v$, and for $k$ fractional divisorial ideals $\mathfrak{a}_1,\dots,\mathfrak{a}_k$ define their $v$-product to be $$\vprod_{i=1}^k\mathfrak{a}_i:=\Bigl(\prod_{i=1}^n\mathfrak{a}_i\Bigr)_v.$$ By convention an empty product is equal to the trivial ideal $A$.
	A fractional divisorial ideal $\mathfrak{a}\in \mathcal{F}_v(A)$ is  \defit{v-invertible} if $\mathfrak{a}\cdot_v\mathfrak{a}^{-1}=A.$ Every non-zero principal fractional ideal $xA$ is invertible with inverse $x^{-1}A$. The group $\mathcal{F}_v(A)^\times$ is the group of $v$-invertible fractional ideals. We have that $(xA)\cdot_v (yA)=xyA,$ hence the subset $\mathcal{H}(A)$ of all non-zero principal fractional ideals is a subgroup of $\mathcal{F}_v(A)^\times.$  So the following definition makes sense.
	\begin{definition}[Class groups]
		Let $A$ be a Krull domain. 
		The \defit{(divisor) class group} of $A$ is the quotient group $\mathcal{C}(A)=\mathcal{F}_v(A)^\times/\mathcal{H}(A).$ 
		We write the group additively. For $\mathfrak{a}\in \mathcal{F}_v (A)^\times$ we denote
			by $[\mathfrak{a}] \in C(A)$ the class containing $\mathfrak{a}$.
	\end{definition}
	
	A domain $A$ is $v$-\defit{noetherian} if it satisfies the ascending chain condition on divisorial ideals.
	An element $x \in \mathbf{q}(A)$ is \defit{almost integral} if there exists $c \in \mathbf{q}(A)^\bullet $ such that $cx^n\in A$ for all $n\ge 0$. A domain $A$ is \defit{completely integrally closed} if $A=\{\,x\in\mathbf{q}(A)\mid x \, \text{almost integral}\,\}$. Every noetherian domain is $v$-noetherian and a noetherian domain is integrally closed if and only if it is completely integrally closed.
	The following theorem is a characterization of Krull domains in terms of multiplicative ideal theory.
	
	\begin{theorem}[{\cite{FOSSUM}}]\label{thm:krull}
		Let $A$ be a domain. Then $A$ is a Krull domain if and only if $A$ is completely integrally closed and $v$-noetherian.
	\end{theorem}
	
	Dedekind domains are exactly the one-dimension Krull domains \cite[Theorem 2.10.6]{GH06}. In a Dedekind domain every non-zero proper ideal factors into a product of prime ideals. Krull domains generalize this property of Dedekind domains, as shown in the next theorem.
	\begin{theorem}[{\cite[Corollary 3.14]{FOSSUM}}]\label{thm:valkrull}
		Let $A$ be a Krull domain. Then every invertible fractional divisorial ideal $\mathfrak{a}\in \mathcal{F}_v(A)^\times$ has a representation as divisorial product $$\mathfrak{a}=\vprod_{\mathfrak{p}\in \mathfrak{X}(A)} \mathfrak{p}^{n_{\mathfrak{p}}},$$ with uniquely determined $n_{\mathfrak{p}}\in \ZZ,$ almost all of which are 0. We have $a\in \mathcal{I}_v(A)^\bullet$ if and only if $n_{\mathfrak{p}}\ge 0$ for all $\mathfrak{p}\in \mathfrak{X}(A)$.
	\end{theorem}
	For $\mathfrak{a}\in \mathcal{F}_v(A)^\times$ we define the $\mathfrak{p}$-\defit{adic valuation} of $\mathfrak{a}$ as $v_{\mathfrak{p}}(\mathfrak{a})=n_{\mathfrak{p}}$ with $n_{\mathfrak{p}}$ as in the previous theorem.
	For $x\in \mathbf{q}(A)$ one has $x\in A$ if and only if $v_{\mathfrak{p}}(x)=v_{\mathfrak{p}}(xA)\ge 0$ for every $\mathfrak{p}\in \mathfrak{X}(A)$.
	
	Krull domains theory and factorization theory are intimately related by the following theorem.
	
	\begin{theorem}[\cite{GH06}]
		Let $A$ be a domain. Then the following are equivalent:
		\vspace{0.1cm}
		\begin{enumerate}[label=\textup(\alph*\textup)]
			\item $A$ is factorial;
			\item $A$ is atomic and every atom is a prime element;
			\item $A$ is a Krull domain and $\mathcal{C}(A)=0$.
		\end{enumerate}  
	\end{theorem}

	 Krull domains have the following important property, which we will use later to prove our main theorem (cfr. Section \ref{sec3}). 
	
	\begin{theorem}[Approximation property, \cite{FOSSUM}]\label{thm:approx}
		Let $A$ be a Krull domain. For all $n\in \NN$, pairwise distinct $\mathfrak{p}_1,\dots,\mathfrak{p}_n \in \mathfrak{X}(A)$ and integers $e_1,\dots,e_n\in\ZZ$, there exists an element $x\in \mathbf{q}(A)$ such that $v_{\mathfrak{p}_i}(x)=e_i$ and $v_\mathfrak{p}(x)\ge 0$ for every $\mathfrak{p}\in \mathfrak{X}(A)\setminus\{\mathfrak{p}_1,\dots,\mathfrak{p}_n\}.$
	\end{theorem}

	\section{Factorization theory in upper cluster algebras}\label{sec2}
	In this section, we prove that every cluster variable is a strong atom and that every (upper) cluster algebra is an FF-domain.
	Recall the following.
	
	\begin{theorem}[\cite{CKQ22}]\label{thm:unit}
		Let $\Sigma=(\x,\y,B) $ be a seed. Let $\A=\A(\Sigma)$,~$\UU=\UU(\Sigma)$ be the cluster algebra and upper cluster algebra associated to $\Sigma$, respectively. Then
		\begin{enumerate}
			\item \begin{enumerate}
				\item $\A^\times=\{\,\lambda x_{n+1}^{a_{n+1}}\dots x_{n+m}^{a_{n+m}}\,\mid \, \lambda\in K^\times,\, a_i\in \ZZ\,\}$;
				\item every cluster variable is an atom of $\A$ \textup{(}pairwise non-associated\textup{)}.
			\end{enumerate}
			\item \begin{enumerate}
				\item $\UU^\times=\{\,\lambda x_{n+1}^{a_{n+1}}\dots x_{n+m}^{a_{n+m}}\,\mid \, \lambda\in K^\times,\, a_i\in \ZZ\,\}$;
				\item every cluster variable is an atom in $\UU$ \textup{(}pairwise non-associated\textup{)}.
			\end{enumerate}
		\end{enumerate}
	\end{theorem}
    From the existing literature (\cite[Corollary 4.2]{GLS13}, \cite[Corollary 1.23]{GELS19}, and \cite[Theorem 4.9]{CKQ22}), we obtain the following characterization of factorial (upper) cluster algebras.
    We remark the directions \ref{fact:oneprime}$\,\Rightarrow$\ref{fact:fact} (or \ref{fact:allprime}$\,\Rightarrow\,$\ref{fact:fact}) actually make use of the fact that (upper) cluster algebras are atomic, see Remark \ref{rem:atomic} below.
 
	\begin{proposition}\label{prop:nagata}
		Let $\Sigma=(\x,\y,B) $ be a seed. Let $H$ be either $\A(\Sigma)$ or $\UU(\Sigma)$. Then the following are equivalent:
		\begin{enumerate}[label=\textup{(}\alph*\textup{)}]
			\item\label{fact:fact} $H$ is factorial,
			\item\label{fact:allprime}  every cluster variable is a prime of $H$,
			\item\label{fact:oneprime} the variables $x_1,\dots,x_n$ are primes of $H$,
		\end{enumerate}
		If $H=\UU$ and $H$ is factorial, then $H=\bigcap_{i=0}^n\mathcal{L}_{\x_i}$.
	\end{proposition}
	
	To show that cluster variables are not just atoms, but even strong atoms, we need the following.
	
	\begin{lemma}[{\cite[Lemma 5.2]{CKQ22}}]\label{lemma:monomials}
		Let $\UU$ be an upper cluster algebra and $(\x,\y,B)$ be a seed. If $x_1^{a_1}\ldots x_n^{a_n}\in \UU$ for some $a_1,\dots,a_n\in \ZZ$, then $a_1,\dots,a_n\in \NN_0$. 
	\end{lemma}	 
	
	For every subset $S$ of a domain $A$ we denote by $\llbracket S \rrbracket_A$ the \defit{divisor-closed submonoid generated by $S$}, that is the set of all elements that divide an element of the submonoid generated by $S$. In particular if $u\in A$ then 
 \[ \llbracket u \rrbracket_A =\{\,a\in A \, \colon a| u^n, \, \text{for some}\, n\in\NN_0\,\}. \]
    \begin{proposition}\label{prop:divclosclusters}
		Let $\A$ be a cluster algebra and $\UU$ be the corresponding upper cluster algebra. Let $(\x,\y,B)$ be an arbitrary seed. Then the divisor-closed submonoid generated by a cluster monomial in $\x$ is $$\llbracket x_1^{e_1}\cdots x_n^{e_n} \rrbracket_{\UU} =\llbracket x_1^{e_1}\cdots x_n^{e_n} \rrbracket_{\A}= \{\,\epsilon x_1^{a_1}\cdots x_n^{a_n}\mid \epsilon\in \UU^\times,\, a_i\in \NN_0, \, \text{with} \, \phantom{,} a_i=0 \, \phantom{,}  \text{if}\, \phantom{,}  e_i=0\,\},$$ for all $e_1,\dots,e_n\in \NN_0$. 
		In particular, the monoid $\llbracket x_1^{e_1}\cdots x_n^{e_n} \rrbracket_{\UU} =\llbracket x_1^{e_1}\cdots x_n^{e_n} \rrbracket_{\A}$ is free abelian.
	\end{proposition}
	\begin{proof}
		Denote by $M$ the set $M=\{\,\epsilon x_1^{a_1}\cdots x_n^{a_n}\mid \epsilon\in \UU^\times,\, a_i\in \NN_0, \, \text{with} \, \phantom{,} a_i=0 \, \phantom{,}  \text{if}\, \phantom{,}  e_i=0\,\}.$ 
		
		First, we prove that $M \subseteq \llbracket x_1^{e_1}\cdots x_n^{e_n} \rrbracket_{\A}$. Let $a\in M$, say $a=\epsilon x_1^{a_1}\cdots x_n^{a_n}$ for some $a_i\in \NN_0$,~$\epsilon \in \UU^\times$. Then $a$ divides the $N^{\text{th}}$-power of $x_1^{e_1}\cdots x_n^{e_n}$ with $N=\max\{\lceil \frac{a_i}{e_i} \rceil,\, i\in [1,n],\, e_i\ne 0 \}$, since trivially  $$(x_1^{e_1}\cdots 	x_n^{e_n})^N=(\epsilon x_1^{a_1}\cdots x_n^{a_n})(\epsilon^{-1}x_1^{e_1N-a_1}\cdots x_n^{e_nN-a_n})$$ holds. Hence $a\in \llbracket x_1^{e_1}\cdots x_n^{e_n} \rrbracket_{\A}$.
		
		Now, we show that $\llbracket x_1^{e_1}\cdots x_n^{e_n} \rrbracket_{\UU}\subseteq M$. Let $a\in \llbracket x_1^{e_1}\cdots x_n^{e_n} \rrbracket_{\UU}$, so there exist a non-negative integer $N\in \NN_0$ and an element $b\in \UU$ such that \begin{equation}\label{eq:3}
	 			x_1^{e_1N}\cdots x_n^{e_nN}=ab.
		\end{equation}
		Consider the Laurent expansion of $a$ and $b$ with respect to $\x$, say \[a=\frac{P(x_1,\dots,x_{n+m})}{x_1^{\alpha_1}\cdots x_{n+m}^{\alpha_{n+m}}},\qquad b=\frac{Q(x_1,\dots,x_{n+m})}{x_1^{\beta_1}\cdots x_{n+m}^{\beta_{n+m}}},\] where $P$,~$Q\in K[\x,\y]$ and $\alpha_i$,~$\beta_i\in \NN_0$ for every $i\in [1,n+m]$.
		
		So Equation \eqref{eq:3} can be rewritten as  $$x_1^{\alpha_1+\beta_1+e_1N}\cdots x_n^{\alpha_n+\beta_n+e_nN}x_{n+1}^{\alpha_{n+1}+\beta_{n+1}}\cdots x_{n+m}^{\alpha_{n+m}+\beta_{n+m}}=PQ,$$ hence $P$ and $Q$ must be monomials in $x_1,\dots,x_{n+m}$, and a fortiori $a$,~$b$ are associated to Laurent monomials in $x_1,\dots,x_{n+m}.$ By Lemma \ref{lemma:monomials}, we get $a=\epsilon x_1^{a_1}\cdots x_{n}^{a_{n}}$ with $a_1,\dots,a_n\in \NN_0$ and $\epsilon\in\UU^\times.$ Observe that, if there is $i\in [1,n]$ such that $e_i=0,$ then $x_i \nmid a,$ whence $a_i=0.$
		
		Therefore, we have proved the following inclusions 
		$$M \subseteq \llbracket x_1^{e_1}\cdots x_n^{e_n} \rrbracket_{\A}\subseteq \llbracket x_1^{e_1}\cdots x_n^{e_n} \rrbracket_{\UU} \subseteq M,$$ and this concludes the proof. 
	\end{proof}
	
	\begin{corollary}
		Let $\Sigma=(\x,\y,B) $ be a seed. Let $\A=\A(\Sigma)$and $\UU=\UU(\Sigma)$ be the cluster algebra and upper cluster algebra associated to $\Sigma$, respectively. Then every cluster variable is a strong atom of $\A$ and $\UU$.
	\end{corollary}
	\begin{proof}
		The statement is a direct consequence of Proposition \ref{prop:divclosclusters} applied to the divisor-closed submonoid generated by one cluster variable.
	\end{proof}

	The Laurent phenomenon implies that $ A\subseteq \UU\subseteq \bigcap_{k=0}^n\mathcal{L}_{\x_k}$, for every seed $(\x,\y,B)$. The ring $\bigcap_{k=0}^n\mathcal{L}_{\x_k}$ is a finite intersection of Laurent polynomial rings, which, being factorial, are in turn intersections of DVRs. So we can write $\bigcap_{k=0}^n\mathcal{L}_{\x_k}=\bigcap_{i\in I}D_i$, with $D_i$ a DVR for every $i\in I$. 	
		Hence we have a monoid homomorphism
		$$\mathbf{v}\colon \UU\to \NN_0^{(I)},  \qquad u\mapsto (v_i(u))_{i\in I}\,, $$ where $v_i\colon \mathbf{q}(D_i)\to\ZZ$ is the discrete valuation of $D_i$ and $\NN_0^{(I)}=\{\,(a_i)\in \NN_0^I\mid a_i=0 \,\, \text{for all but}$ $\text{fi\-ni\-te\-ly many}\, i\in I\,\}$. Observe that, by Theorem \ref{thm:unit}, \begin{equation}\label{eq:1}
			\A^\times=\UU^\times=\Bigl(\bigcap_{k=0}^n\mathcal{L}_{\x_k}\Bigr)^\times=\Bigl(\bigcap_{i\in I}D_i\Bigr)^\times=\bigcap_{i\in I}D_i^\times.
		\end{equation}

	\begin{lemma}\label{lemma:divhom}
		Let $\Sigma=(\x,\y,B)$ be a seed and $\A=\A(\Sigma)$ and $\UU=\UU(\Sigma)$ be the cluster algebra and the upper cluster algebra associated to $\Sigma$, respectively. Let $\mathbf{v}\colon \UU\to \NN_0^{(I)}$ be the monoid homomorphism defined above. Then  $$\mathbf{v}(a)=\mathbf{v}(b) \iff a\simeq_{\,\UU} b\iff a\simeq_\A b.$$
	\end{lemma} 
	\begin{proof}
		First assume that $a$,~$b$ are two elements of $\UU$ such that $\mathbf{v}(a)=\mathbf{v}(b)$, say $v_i(a)=v_i(b)=n_i$ for every $i\in I$.  Let $p_{i}$ denote the unique (up to associates) prime element of $D_{i}$ with $i\in I$. By definition of a DVR (see Section \ref{subsec:krull}), for every $i\in I$, we have that $a=u_ip_i^{n_i}$ and $b=u_i'p_i^{n_i}$ with $u_i$,~$u'_i\in D_i^\times$, that is $a=w_ib$ for some $w_i\in D_i^\times$. Hence $w_ib=w_jb$ for all $ i$,~$j\in I$ implies $w_i=w_j$ and so $a=wb$ for some $w\in \bigcap_{i\in I} D_i^\times$. Thus Equation \eqref{eq:1} implies that $a\simeq_{\,\UU} b$, and this is equivalent to $a\simeq_{\A} b$.
		The other direction is straightforward.
	\end{proof}

	\begin{theorem}\label{prop:FF}
		Let $\Sigma=(\x,\y,B)$ be a seed and $\A=\A(\Sigma)$ and $\,\UU=\UU(\Sigma)$ be the cluster algebra and the upper cluster algebra associated to $\Sigma$, respectively. Then $\A$ and $\UU$ are {\rm FF}-domains. In particular, they are {\rm BF}-domains and atomic.
	\end{theorem}
	\begin{proof}
		Let $H$ denote either $\A^\bullet$ or $\UU^\bullet$, and let $a\in H$. We want to show that $a$ has only finitely many non-associated divisors. 
		Consider the set $\Omega=\{\,e_i\mid v_i(a)>0\,\}$ of all the atoms of $\NN_0^{(I)}$ that divide $\mathbf{v}(a)$. Here $e_i$ denotes the tuple with all components $0$ except the $i$-th entry that is $1$. Notice that $\Omega$ is finite.
		
		Let $v\in H$ be such that $v| a$. Then, since $\mathbf{v}(v)|\mathbf{v}(a)$, there exist $\omega_1,\dots,\omega_k\in\Omega$ such that $\mathbf{v}(v)=\omega_1+\cdots+\omega_k,$ hence, since $\Omega$ is finite, there are only finitely many possibilities for $\mathbf{v}(v)$ for each divisor $v$ of $a$. So let $u$ be another divisor of $a$ such that $\mathbf{v}(u)=\omega_1+\cdots+\omega_k=\mathbf{v}(v).$ Then Lemma \ref{lemma:divhom} implies that $u\simeq_H v$, thus $a$ has only finitely many non-associated divisors.
	\end{proof}
	
	\begin{remark} \label{rem:atomic}
        In \cite{GELS19} the cluster algebras under investigation are usually Krull domains, and hence FF-domains (in particular, atomic).
        However, \cite[Corollary 1.23]{GELS19} (which is subsumed in our Proposition~\ref{prop:nagata}), concerns arbitrary cluster algebras. 
        To deduce that the cluster variables being prime is sufficient for an (upper) cluster algebra to be factorial, one needs to know a priori that the cluster algebra is atomic.
        This was taken for granted in \cite[Corollary 1.23]{GELS19}.
        However, it is somewhat non-trivial and our Theorem~\ref{prop:FF} provides a proof of this fact.
		More generally, in the statement \cite[Corollary 1.20]{GELS19} of a corollary of Nagata's Theorem, an assumption that $A$ be atomic is missing.
        See \cite[Section 1]{AAZ92} for details.
        A proof that (upper) cluster algebras are atomic is published in \cite[Appendix A]{CKQ22}. In fact, this proof shows that they are BF-domains (but not the stronger claim that they are FF-domains).
	\end{remark}
	
	Recall that a domain is a Krull domain if and only if it is completely integrally closed and $v$-noetherian. Moreover, a Krull domain is always an FF-domain. 
	\begin{theorem}\label{prop:cickrull}
		Let $\Sigma=(\x,\y,B)$ be a seed and $\,\UU=\UU(\Sigma)$ be the upper cluster algebra associated to $\Sigma$. Then $\UU$ is completely integrally closed. 
		Moreover, if there exists a seed such that $\,\UU$ satisfies the starfish condition at that seed, then $\UU$ is a Krull domain.
	\end{theorem}
	\begin{proof}
		The upper cluster algebra $\UU$, being an intersection of completely integrally closed domains, is completely integrally closed. Moreover, if we assume $\UU=\bigcap_{i=0}^n\mathcal{L}_{\x_i}$ for some seed $(\x,\y,B)$, then $\UU$, being a finite intersection of Krull domains, is a Krull domain.
	\end{proof}
	
	\section{Class Groups of Upper Cluster Algebras}\label{sec3}
    In this section we determine the class groups of full rank upper cluster algebras (more generally, upper cluster algebras, satisfying the starfish condition).
	To do so, we need some preliminary results.
    We prove them in the more general setting of upper cluster algebras that are Krull domains.
    \begin{proposition}\label{prop:valuationlaurentpolyn}
		Let $A$ be a domain and let  $x_1,\dots,x_n\in A$ be such that $A_{\x}:=A[x_1^{-1},\dots,x_n^{-1}]=D[x_1^{\pm 1},\dots,x_n^{\pm 1}]$ is a factorial Laurent polynomial ring for some subring $D$ of $A$. If $f\in A$ is such that $f\notin A_{\x}^\times$ and $f$ has no repeated factors in $A_{\x}$, then there exist a height-$1$ prime ideal $\mathfrak{p}$ of $A$  and a discrete valuation $v_{\mathfrak{p}}\colon \mathbf{q}(A)\to \ZZ\cup\{\infty\}$ such that $v_\mathfrak{p}(f)=1$ and $v_\mathfrak{p}(x_i)=0$ for $i\in [1,n].$
	\end{proposition}
	\begin{proof}
		Let $p\in A_\x$ be a prime factor of $f$. Set $\mathfrak{p}=pA_\x\cap A.$ The prime ideal $pA_\x$ has height $1$ by Krull's Principal Ideal Theorem and hence so does $\mathfrak{p}$ (Proposition \ref{prop:loc.ideals}). Since $f$ has no repeated factors, $v_{\mathfrak{p}}(f)=1$. For $i\in [1,n]$,  $x_i$ is a unit of $A_\x$, therefore $x_i\notin \mathfrak{p}$. Hence $v_\mathfrak{p}(x_i)=0.$
	\end{proof}
	Notice that for the proof a weaker assumption is sufficient. It is enough to assume that there exists $f\notin A_\x^\times$ that has a prime factor of multiplicity $1$. 
	\begin{corollary}\label{cor:valuationsuppclusalg}
		Let $\Sigma=(\x,\y, B)$ be a seed. Let $\UU=\UU(\Sigma)$ be the upper cluster algebra associated to $\Sigma$. Then, for every $i \in I$ there exists a height-$1$ prime ideal $\mathfrak{p}$ of $\UU$ such that $v_\mathfrak{p}(x_i)=1$ and $v_\mathfrak{p}(x_j)=0$ for every $j\in [1,n]\setminus \{i\}$.
	\end{corollary}
	\begin{proof}
		We mutate $\x=(x_1,\dots,x_n)$ in direction $i$ obtaining a new seed $\x_i.$ Let $f_i=x_ix_i'\in K[\x_i,\y]$ be the exchange polynomial of $x_i$ associated to $\Sigma.$ We have that $$\UU[\x_i^{- 1},\y^{- 1}]=K[\x_i^{\pm1},\y^{\pm 1}],$$ hence $K[\x_i^{\pm1},\y^{\pm 1}]$ is a Laurent polynomial ring, and therefore factorial. By \cite[Proposition 2.3]{GELS19} we know that $f_i$ does not have repeated factors. Moreover, due to our assumption on isolated seeds (Remark \ref{rmk:isolated}), one has $f_i\notin K[\x_i^{\pm 1},\y^{\pm 1}]^\times$,  so $f_i$ satisfies the hypothesis of Proposition \ref{prop:valuationlaurentpolyn}. Hence we can conclude that there exists a height-$1$ prime ideal $\mathfrak{p}$ of $\UU$ such that $v_\mathfrak{p}(f_i)=1$, $v_\mathfrak{p}(x_j)=0$ for every $j\in[1,n]\setminus\{i\}$, and $v_\mathfrak{p}(x_i')=0$, hence $v_\mathfrak{p}(x_i)=1$. 
	\end{proof}

	\smallskip
	
	Let us recall this very general result on Krull domains.
	\begin{theorem}[{\cite[Theorems 3.1 and 3.2]{GELS19}}]\label{thm:classgroupgeneralcase}
		Let $A$ be a Krull domain, and let $x_1,\dots,x_n\in A$ be such that $A_{\x}:=A[x_1^{-1},\dots,x_n^{-1}]=D[x_1^{\pm 1},\dots,x_n^{\pm 1}]$ is a factorial Laurent polynomial ring for some subring $D$ of $A$. Let $\mathfrak{p}_1,\dots,\mathfrak{p}_t$ be the pairwise distinct height-1 prime ideals of $A$ containing one of the elements $x_1,\dots,x_n$. Suppose that $$x_iA={\vprod_{l=1}^t} \mathfrak{p}_j^{a_{ij}},$$ with $\mathbf{a}_i=(a_{ij})_{j=1}^t\in \NN_0^t.$ Then $\mathcal{C}(A)\cong \ZZ^t/\langle\mathbf{a}_i\mid i\in[1,n]\rangle$ and it is generated by $[\mathfrak{p}_1],\dots,[\mathfrak{p}_t]$.
		
		Suppose in addition that $D$ is infinite and either $n\ge 2$ or $n=1$ and $D$ has at least $|D|$ height-1 prime ideals. Then every class of $\mathcal{C}(A)$ contains precisely $|D|$ height-1 prime ideals.
	\end{theorem}
	
	We are now ready to prove our main theorem on the class groups of upper cluster algebras that are Krull domains.
	
	\begin{theorem}\label{thm:classgroupupper}
		Let $\Sigma=(\x,\y, B)$ be a seed and let $\UU=\UU(\Sigma)$ be the upper cluster algebra associated to $\Sigma$. 
		Suppose that $\UU$ is a Krull domain. Let $t\in \NN_0$ denote the number of height-$1$ prime ideals that contain one of the exchangeable variables $x_1,\cdots,x_n$. Then the class group $\mathcal{C}(\UU)$ of $\UU$ is a free abelian group of rank $t-n$. 
		
		In particular, each class contains exactly $|K|$ height-1 prime ideals.
	\end{theorem}
	\begin{proof}       
		The proof follows the strategy of \cite[Theorem A]{GELS19} where they proved the statement for cluster algebras that are Krull domains.
		
		In order to apply Theorem \ref{thm:classgroupgeneralcase}, notice that the Laurent phenomenon implies that $$\UU[\x^{-1},\y^{-1}]=K[\x^{\pm 1},\y^{\pm 1}]$$ and $K[\x^{\pm 1},\y^{\pm 1}]$ is a factorial domain. $\UU$ is a Krull domain by assumption, hence any principal ideal is a divisorial product of height-$1$ prime ideals, in particular, for every $i\in [1,n]$ there exist $a_{i1},\dots,a_{it}\in \NN_0$ such that \begin{equation}\label{eq:2}x_i\UU={\vprod_{l=1}^t} \mathfrak{p}_j^{a_{ij}}, \end{equation} with $\mathfrak{p}_1,\dots, \mathfrak{p_t}$ pairwise distinct height-$1$ prime ideals of $\UU$ that contain one of the elements $x_1,\dots,x_n$ and $a_{ij}=v_{\mathfrak{p}_j}(x_i)$. Thus Theorem \ref{thm:classgroupgeneralcase} implies that $C(\UU)$ is generated by $[\mathfrak{p}_1],\dots,[\mathfrak{p}_t]$. 
		
		By Corollary \ref{cor:valuationsuppclusalg}, for every $i\in [1,n]$ there exists $\mathfrak{p}\in \mathfrak{X}(\UU)$ such that $v_{\mathfrak{p}}(x_i)=1$ and $v_{\mathfrak{p}}(x_j)=0$ for every $j\ne i$. Thus, since $\mathfrak{p}_1,\dots, \mathfrak{p_t}$ are all the height-1 prime ideals that contain $x_i$, there exists $k_i\in [1,t]$ such that $v_{\mathfrak{p}_{k_i}}(x_i)=a_{ik_i}=1$ and $v_{\mathfrak{p}_{k_i}}(x_j)=a_{jk_i}=0$ for every $j\in [1,n]\setminus \{i\}$. Hence, by \eqref{eq:2}, $$0=a_{i1}[\mathfrak{p}_{1}]+\cdots+[\mathfrak{p}_{k_i}]+\cdots+a_{it}[\mathfrak{p}_{t}],$$  that is $[\mathfrak{p}_{k_i}]$ is a linear combination of $[\mathfrak{p}_1],\dots,[\mathfrak{p}_t]$. If $i,j\in[1,n]$ with $j\ne i$, then $\mathfrak{p}_{k_i}\ne \mathfrak{p}_{k_{j}}.$ Indeed, if they were equal,  Corollary \ref{cor:valuationsuppclusalg} would imply $1=a_{ik_i}=a_{ik_{j}}=0$ and this would be a contradiction. Thus $\mathfrak{p}_{k_1},\dots,\mathfrak{p}_{k_n}$ are $n$ superfluous generators and $C(\UU)$ is a free abelian group generated by $t-n$ elements. 
		
		If $n+m\ge 2$, or $n+m=1$ and $K=\ZZ$, we can apply Theorem \ref{thm:classgroupgeneralcase} to obtain that every class contains exactly $|K|$ height-1 prime ideals. Suppose then $n+m=1$ and $K$ is a field. Since we assumed there is no isolated exchangeable index, necessarily $n=0$ and $m=1$. Then $\UU=K[x_1^{\pm 1}]$, $\mathcal{C}(\UU)=0$ and $\UU$ contains $|K|$ pairwise non-associated prime elements.
	\end{proof}
	
	This theorem leads us to a dichotomy between factorial upper cluster algebras and non-factorial ones.
	Let $A$ be a domain and $a \in A^\bullet$.
	We call $k \ge 0$ a \defit{length} of $a$ if there exist atoms $u_1,\dots,u_k \in A^\bullet$ such that $a=u_1\cdots u_k$.
	The \defit{length set} of $a$, denoted by $\mathsf L(a)$, is the set of all such lengths; we set $\mathsf L(a)=\{0\}$ for $a \in A^\times$.
	Then $\mathsf L(a) = \{0\}$ if and only if $a$ is a unit, and $\mathsf L(a) = \{1\}$ if and only if $a$ is an atom. In a Krull domain $\mathsf L(a)$ is always a finite set.
	
	\begin{corollary}\label{cor:setlenghts}
		Let $\UU$ be an upper cluster algebra. Assume that $\UU$ is a Krull domain.
		
		\begin{itemize}
			\item If $\UU$ is factorial, then $\mathsf L(u)$ is a singleton for each non-zero element $u \in \UU.$
			\item If $\UU$ is not factorial, then for every finite set $L \subseteq \ZZ_{\ge 2}$ there exists an element $u \in \UU^\bullet$ such that $\mathsf L(u) = L$.
		\end{itemize}		 
	\end{corollary}
	\begin{proof}
		If $\UU$ is factorial the claim is trivial. If $\UU$ is not factorial, the claim follows by Theorem \ref{thm:classgroupupper} and a result of Kainrath \cite[Theorem 1]{K99}.
	\end{proof}
	
Now we compute the rank of the class group of a full rank upper cluster algebra in terms of the irreducible factors of the exchange polynomials.
		
		 Let $\Sigma=(\x,\y,B)$ be a seed and $\UU=\UU(\Sigma)$ be the upper cluster algebra associated to $\Sigma.$ Recall that, by definition, the upper cluster algebra $\UU$ is an intersection of some $\mathcal{L}_\mathbf{z}$, where $\mathcal{L}_\mathbf{z}=K[u^{\pm 1}\mid u\in \mathbf{z}\cup \y]$  is the localization of $\UU$ to the set $S_\mathbf{z}=\{\,z_{1}^{a_1}\cdots z_n^{a_n} x_{n+1}^{a_{n+1}}\cdots x_{n+m}^{a_{n+m}}\mid a_i\in \NN _0\,\},$ where $\mathbf{z}\sim \x$. On the other hand if we suppose that $\UU$ is a Krull domain, then $\UU=\bigcap_{\mathfrak{p}\in\mathfrak{X}(\UU)}\UU_{\mathfrak{p}}.$ Moreover, the following equalities holds $$ \UU=\bigcap_{\mathbf{z}\sim \x}\mathcal{L}_{\mathbf{z}}=\bigcap_{\mathbf{z}\sim \x}\bigcap_{\mathfrak{p}\in \mathfrak{X}(\mathcal{L}_{\mathbf{z}})}(\mathcal{L}_\mathbf{z})_\mathfrak{p}=	\bigcap_{\mathbf{z}\sim \x}\bigcap_{\substack{\mathfrak{q}\in \mathfrak{X}(\UU)\\ \mathfrak{q}\cap S_\mathbf{z}= \emptyset}}\UU_\mathfrak{q},$$ where the last equality follows from Proposition \ref{prop:loc.ideals}, since $\mathcal{L}_\mathbf{z}$ is the localization of $\UU$ at $S_{\mathbf{z}}.$ 
	
	\begin{theorem}\label{thm:primeidealupper}
		Let $\Sigma=(\x,\y,B)$ be a seed and $\UU=\UU(\Sigma)$ be the upper cluster algebra associated to $\Sigma.$ Suppose in addition that $\UU$ is a Krull domain. Then for every height-1 prime ideal $\mathfrak{p}$ of $\UU$ there exists a seed $(\mathbf{z},\y,C)$  such that $\mathfrak{p}\mathcal{L}_\mathbf{z}\in \mathfrak{X}(\mathcal{L}_\mathbf{z}).$
		
		Furthermore, if $\mathcal{U}=\bigcap_{i=0}^{n}\mathcal{L}_{\mathbf{z}_i}$ for some seed $(\mathbf{z},\y,C)$, then for every height-1 prime ideal $\mathfrak{p}$ of $\UU$ there exists $k\in [0,n]$ such that $\mathfrak{p}\mathcal{L}_{\mathbf{z}_k}\in \mathfrak{X}(\mathcal{L}_{\mathbf{z}_k}).$
	\end{theorem}
	
	\begin{proof} We prove the first claim, the second is completely analogous. 
		Proceed by contradiction, i.e., suppose that there exist $\mathfrak{p}_{0}\in\mathcal{X}(\UU)$ such that $\mathfrak{p}_{0}\mathcal{L}_\mathbf{z}\notin \mathfrak{X}(\mathcal{L}_\mathbf{z})$ for every seed $(\mathbf{z},\y,C).$ We claim that there exists an element $a\in \UU$ such that $v_{\mathfrak{p}_{0}}(a)<0.$
		
		Let $y\in \mathfrak{p}_{0},$ and let $\mathfrak{p}_1,\dots,\mathfrak{p}_r\in \mathfrak{X}(\UU)$ be such that $\{\mathfrak{p}_{0},\mathfrak{p}_{1},\dots,\mathfrak{p}_{r}\} $ is the set of all the distinct height-1 prime ideals of $\UU$ that contain $y$. Notice that $v_{\mathfrak{p}_i}(y)>0$ for every $i\in [0,r]$ and $v_{\mathfrak{q}}(y)=0$ for every $\mathfrak{q}\ne \mathfrak{p}_i$. 
		Set $e_i:=v_{\mathfrak{p}_i}(y)\in \NN$ for every $i\in[1,r]$ and $e_{0}=0$. By the Approximation property (Theorem \ref{thm:approx}) there exists an element $x\in \UU$ such that $ v_{\mathfrak{p}_i}(x)= e_i$ for every $i\in [0,r]$ and $v_\mathfrak{q}(x)\ge 0$ for every $\mathfrak{q}\in \mathfrak{X}(\UU)\setminus \{\mathfrak{p}_0,\dots,\mathfrak{p}_r\}.$ 
		
		Define now $a:=x/y\in \mathbf{q}(\UU).$ By construction  then $$v_{\mathfrak{p}_{0}}(a)=v_{\mathfrak{p}_{0}}(1/y)<0, \,\,v_{\mathfrak{p}_{i}}(a)=v_{\mathfrak{p}_{i}}(x)-v_{\mathfrak{p}_{i}}(y)=0$$  for every $i\in[1,r]$ and $v_\mathfrak{q}(a)=v_\mathfrak{q}(x)\ge 0$ for every $\mathfrak{q}\in \mathfrak{X}(\UU)\setminus\{\mathfrak{p}_{0},\mathfrak{p}_{1},\dots,\mathfrak{p}_{r}\}.$ Hence this implies that $a\in \UU_\mathfrak{q}$ for every $\mathfrak{q} \in \mathfrak{X}(\UU)\setminus \{\mathfrak{p}_0\}$ and  $a\notin \UU_{\mathfrak{p}_0}$, hence $a\notin \UU$. On the other hand, by assumption $\mathfrak{p}_0\cap S_\mathbf{z}\ne \emptyset$ for every seed $(\mathbf{z},\y,C)$, hence $$a\in \bigcap_{\mathbf{z}\sim \x}\bigcap_{\substack{\mathfrak{q}\in \mathfrak{X}(\UU)\\ \mathfrak{q}\cap S_\mathbf{z}=\emptyset}}\UU_\mathfrak{q}=\UU,$$ therefore we found our contradiction.
	\end{proof}
	
	The theorem has an important consequence if we consider the case of upper cluster algebras that satisfy the starfish condition at one seed.

	\begin{corollary}\label{cor:primeupper}
		Let $\UU$ be an upper cluster algebra that satisfies the starfish condition at a seed $(\x,\y,B)$. Let $\mathfrak{p}$ be a height-1 prime ideal of $\UU$ that contains $x_i$ for some $i\in[1,n]$. Then $\mathfrak{p}\mathcal{L}_{\x_i}$ is a height-1 one prime ideal of $\mathcal{L}_{\x_i}$. 
		In particular, every height-1 prime ideal of $\UU$ that contains the cluster variable $x_i$ does not contain any element of the set $\{x_1,\dots,x_i',\dots, x_{n+m}\}.$
	\end{corollary}
	
	\begin{proof}
		By Theorem \ref{thm:primeidealupper} there exists $k\in [0,n]$  such that $\mathfrak{p}\mathcal{L}_{\x_k}\in \mathfrak{X}(\mathcal{L}_{\x_k}).$ In particular $\mathfrak{p}\cap \{x_1,\dots,x_k',\dots,x_{n+m}\}=\emptyset.$ Therefore $k$ must be $i$ and $x_j\notin \mathfrak{p}$ for every $j\in [1,n+m]\setminus\{i\}.$ 
	\end{proof}
	
	We are now ready to determine the rank of the class group of an upper cluster algebra that satisfies the starfish condition at one seed.
	\begin{theorem}\label{thm:rankclassgroup}
		Let $\UU$ be an upper cluster algebra that satisfies the starfish condition at a seed $(\x,\y,B)$. For every $i\in [1,n]$ let $l_i$ be the number of irreducible factors of the exchange polynomial $f_i$. Then $$\mathcal{C}(\UU)\cong \ZZ^r, \, \, \text{with}\, \, r=\sum_{i=1}^n l_i - n.$$
	\end{theorem}
	\begin{proof}
		Theorem \ref{prop:cickrull} implies that $\UU$ is a Krull domain, so we can apply Theorem \ref{thm:classgroupupper} and claim that $\mathcal{C}(\UU)$ is a free abelian group of rank $t-n$ with $t$ the number of prime ideals that contain one of the cluster variables $x_1,\dots, x_n$. 		By Corollary \ref{cor:primeupper}, if a height-1 prime ideal of $\UU$ contains a cluster variable $x_i$, then it does not contain $x_j$ for every $j\in [1,n]\setminus\{i\}$, that is $$t=\sum_{i=1}^n |\{\, \mathfrak{p}\in \mathfrak{X}(\UU)\mid x_i\in \mathfrak{p}\,\}|.$$ Fix an index $i\in[1,n]$, and let $r_1,\dots,r_{l_i}\in K[\x_{i},\y]$ be the pairwise non-associated irreducible factors of the exchangeable polynomial $f_i=x_ix_i'\in K[\x_i,\y]. $ We claim that $$l_i=|\{\, \mathfrak{p}\in \mathfrak{X}(\UU)\mid x_i\in \mathfrak{p}\,\}|.$$ 
		
		Let $\mathfrak{p}\in \mathfrak{X}(\UU)$ be a height-1 prime ideal that contains the cluster variable $x_i$. Corollary \ref{cor:primeupper} implies that   $\mathfrak{p}':=\mathfrak{p}\mathcal{L}_{\x_i}$ is a height-1 prime ideal of $\mathcal{L}_{\x_i}$ that contains $x_{i}$. 
		Observe that $f_i\in \mathfrak{p}'$. Since none of the factors of $f_i$ can be a monomial in $\x$, Lemma \ref{lemma:loc.irr} implies that, for all $k\in[1,l_i]$, $r_k$ is also irreducible in $\mathcal{L}_{\x_i} .$ The ideal $\mathfrak{p}'$ is prime, hence necessarily one of the irreducible factors of $f_i$, say $r_k$, must be in $\mathfrak{p}'.$ Therefore, since $\mathfrak{p}'$ has height $1$ and $\mathcal{L}_{\x_i}$ is factorial, we have that $\mathfrak{p}'=r_k\mathcal{L}_{\x_i},$ and hence $\mathfrak{p}=\UU\cap r_k\mathcal{L}_{\x_i}$. This proves that $$|\{\, \mathfrak{p}\in \mathfrak{X}(\UU)\mid x_i\in \mathfrak{p}\,\}|\le l_i.$$
		
		A similar argument shows that $\mathfrak{p}_j:=\UU\cap r_j \mathcal{L}_{\x_i}$ is a height-1 prime ideal of $\UU$ that contains $x_i$, for every  $j\in [1,l_i]$. Assume $\mathfrak{p}_j=\mathfrak{p}_k$ for some $j\ne k$. Then $r_j \mathcal{L}_{\x_i}=r_k \mathcal{L}_{\x_i}$ and hence $r_j\simeq r_k$, that is impossible by assumption, whence \[ l_i=|\{\, \mathfrak{p}\in \mathfrak{X}(\UU)\mid x_i\in \mathfrak{p}\,\}|. \qedhere\]
	\end{proof}
		If $l_i=1$ for every $i\in[1,n]$, we get immediately the following corollary, which is a slight generalization of \cite[Theorem 4.9]{CKQ22}.
\begin{corollary}
			Let $\UU$ be an upper cluster algebra that satisfies the starfish condition at a seed $(\x,\y,B)$. Then $\UU$ is factorial if and only if the exchange polynomials $f_i=x_ix_i'$ are irreducible.  
	\end{corollary}

	Notice we cannot extend this theorem to upper cluster algebras that are just Krull domains. Indeed, consider the (upper) cluster algebra of finite type $A_3.$ The cluster algebra is a Krull domain and we know that is not factorial ($1+x_{2}^{2}=x_{1}x_{1}'=x_{3}x_{3}'$), hence $\mathcal{C}(\UU)\ne 0.$ However,  $\sum_{i=1}^3 l_i -3=0,$ so the theorem does not apply in this case. 
	
	\begin{remark}
		If the seed is in addition acyclic, the cluster algebra and the upper cluster algebra coincide and we could get the conclusions of Theorem \ref{thm:rankclassgroup} as a consequence of the main Theorem of \cite{GELS19}.
		Two indices $i$,~$j\in [1,n]$ are partners if $f_i$ and $f_j$ have a common non-trivial factor. If $\,\mathcal{U}$ satisfies the starfish condition at a seed $(\x,\y,B)$, then no distinct indices are partners. Indeed, first notice that, if $i$ and $j$ are partners, then $(x_1,\dots,x_i',\dots,x_j',\dots,x_n)$ is a seed, since $b_{ij}=b_{ji}=0$ (cf. \cite[Lemma 2.7]{GELS19}). Suppose now, by contradiction, that there exist two partners $i,j\in [1,n]$ with $i< j$. Thus there exists $h\in K[\x,\y]\setminus K[\x,\y]^\times$ such that $f_i=hg_i$ and $f_j=hg_j$ for some $g_i,g_j\in K[\x,\y]$. Hence the element $s=\frac{hg_ig_j}{x_ix_j}\in \bigcap_{k=0}^n\mathcal{L}_{\x_k},$ but $s\notin K[x_1^{\pm 1},\dots,{x_i'}^{\pm 1},\dots,{x_j'}^{\pm 1},\dots,x_n^{\pm 1}]$, and this is a contradiction. Assume in addition that the (upper) cluster algebra $\,\UU$ is acyclic. Then, since the only partner sets are the sets $\{i\}$, $i\in[1,n] $, Theorem A in \cite{GELS19} implies that the rank of the class group of $\,\UU$ is equal to $t-n$ with $t=\sum_{i=1}^n |\{\, \mathfrak{p}\in \mathfrak{X}(\UU)\mid x_i\in \mathfrak{p}\,\}|$.
		
	\end{remark}
	Denote by $\mu_{d}^*(K)$ the set of primitive $d$-th roots of unit in $K$.
	\begin{examples}
		Let us compute the class group in some specific examples.
		\begin{enumerate}
			\item Consider the seed $\Sigma=\left(\left(x_1,x_2,x_3,x_4\right),\emptyset,B\right)$ where $B$ is the matrix
			\begin{equation*}
			B=\begin{pmatrix}
			0 & -1 & 0 & 4 \\
			2 & 0 & 3 & 6  \\
			0 & -3 & 0 & 0 \\
			-4 & -3 & 0 & 0
			\end{pmatrix}.
			\end{equation*}
			$B$ is a full rank skew-symmetrizable matrix (consider $d_1=d_4=2,\, d_2=d_3=1$).
			
			The exchange polynomials associated to $\Sigma$ are:
			\begin{gather*}
			f_1=x_2^2+x_4^4,\quad f_2=x_1x_3^3x_4^3+1,\\
			f_3=x_2^3+1, \quad f_4=x_1^4x_2^6+1.
			\end{gather*}
			The polynomial $f_1$ has 2 factors if $\mu_4^*(K)\ne\emptyset$, otherwise 1, the polynomial $f_3$ has 3 factors if $\mu_6^*(K)\ne\emptyset$, otherwise 2, and the polynomial $f_4$ has 2 factors if $\mu_4^*(K)\ne\emptyset$, otherwise 1. The following table shows the class group of $\UU$ in all the possible cases:
			\renewcommand\arraystretch{1.3}
			\begin{table}[ht]
				\begin{tabular}{c|c|c}
					
					& $\mu_4^*(K)\ne\emptyset$ & $\mu_4^*(K)=\emptyset$ \tabularnewline
					\hline 
					$\mu_6^*(K)\ne\emptyset$ & $\ZZ^4$ & $\ZZ^2$ \tabularnewline
					\hline
					$\mu_6^*(K)=\emptyset$ & $\ZZ^3$ & $\ZZ$ \tabularnewline
				\end{tabular}
			\end{table}
		Notice that $U(\Sigma)$ is not factorial independently from the choice of the field.
			\item Consider the following quiver $\mathcal{Q}$: 
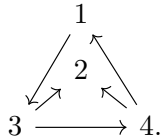
			\begin{equation*}
			\begin{tikzcd}[ampersand replacement=\&, column sep=10 pt, row sep=6 pt]
			\&	1\arrow[ddl] \& \\
			\&  2 		\& \\ 
			3 \arrow[ur] \arrow[rr] \& \& 4. \arrow[uul] \arrow[ul]
			\end{tikzcd}
			\end{equation*}
			The matrix associated to $\mathcal{Q}$ is 
			\begin{equation*}
			B=\begin{pmatrix}
			0 & 0 & 1 & -1 \\
			0 & 0 & -1 & -1  \\
			-1 & 1 & 0 & 1 \\
			1 & 1 & -1 & 0
			\end{pmatrix}
			\end{equation*}
			and the exchange polynomials are:
			\begin{gather*}
			f_1=x_3+x_4,\quad f_2=x_3x_4+1,\\
			f_3=x_2x_4+x_1, \quad f_4=x_1x_2+x_3.
			\end{gather*}
			$B$ has full rank and the polynomials are irreducible, hence $\mathcal{U}(\mathcal{Q})$ is factorial.
	
			\item Consider the seed $\Sigma=((x_1,x_2,x_3),\{x_4\},B)$ where $B$ is the matrix
			\begin{equation*}
			B=\begin{pmatrix}
			0 & 2 & -2  \\
			-2 & 0 & 2 \\
			2 & -2 & 0  \\
			2 & 0 & 0
			\end{pmatrix}.
			\end{equation*}
			$B$ has full rank and the exchange polynomials are
			$$f_1=x_2^2+x_3^2x_4^2,\quad f_2=x_1^2+x_3^2,\quad f_3=x_1^2+x_2^2.$$
            Hence, if $\mu_4^*(K)\ne\emptyset$, then $\mathcal{C}(\mathcal{U})\cong \ZZ^3$, otherwise $\mathcal{U}$ is factorial.
            
			Observe that the exchangeable part of the quiver $\Gamma(B)$, being the Markov quiver, is not acyclic, hence we could not have applied Theorem B in \cite{GELS19}.
			
		\end{enumerate}
		
	\end{examples}
	
	\section{Valuation Pairing}\label{sec4}
	In \cite{CKQ22}, the authors introduced the notion of a \emph{valuation pairing} on an upper cluster algebra and proved a local unique factorization for full rank upper cluster algebras. In this section we give an interpretation of the valuation pairing in terms of the $\mathfrak{p}$-adic valuation in Krull domains.
	
	\smallskip
	
	\begin{definition}[Valuation pairing,\,\cite{CKQ22}]
		Let $\UU$ be an upper cluster algebra and $\mathcal{X}$ be the set of cluster variables of $\UU.$ Define $$(-|-)_v\colon \mathcal{X}\times \UU \to \NN \cup \{\infty\}\qquad (x,u)\mapsto (x\mid u)_v:=\text{max}\{\,s\in \NN\mid u/x^s \in \UU.\,\}$$
	\end{definition}
	
	\smallskip
	
	Clearly, $(x\mid u)_v=0$ if and only if $x\nmid u.$ If $r=(x\mid u)_v>0$, then $u/x^r\in \UU$ and $u/x^{r+1}\notin \UU$, that is there exists $y\in \UU$ such that $u=x^ry$ and $x \nmid y$.
	
	\begin{definition}[Local Factorization,\,\cite{CKQ22}]
		Let $\UU$ be an upper cluster algebra and $u\in \UU$.  We say that $u=ab $ is a \defit{local factorization} of $u$ with respect to the seed $(\x,\y,B)$ if $a$ is a monomial in $\x$ and $(x_i\mid b)_v=0$ for every $i\in[1,n].$ 
	\end{definition}
	
	Using valuation pairings, Cao, Keller, and Qin proved the following.
	\begin{theorem}[{\cite[Propisition 3.5 and Theorem 3.7]{CKQ22}}]
		Let $\UU$ be an upper cluster algebra and $(\x,\y,B)$ any seed. Then every $0\ne u\in \UU$ admits a local factorization with respect to $(\x,\y,B)$.
		Moreover, if $\UU$ has full rank, then this local factorization is unique \textup(up to associates\textup). 
	\end{theorem}

    Let $\UU$ be an upper cluster algebra.
    We can interpret the valuation pairing in terms of discrete valuations on $\UU$.
    By definition, $\UU$ is an intersection of (possibly infinitely many) Laurent polynomial rings.
    Each of these Laurent polynomial rings gives rise to a family of discrete valuations, arising from the height-1 prime ideals of the Laurent polynomial ring.
    Let $\{\,v_i :i \in I\,\}$ be the set of all such discrete valuations arising from all the Laurent polynomial rings.

    Let $x\in \mathcal{X}$ and let $V \subseteq \{\, v_i : i \in I\,\}$ the subset of all valuations $v_i$ for which $v_i(x)>0$. Now consider an element $u \in \UU^\bullet$ and set $r=(x\mid u)_v.$ 
	Let us write $u=x^ry$ with $y\in\UU$ such that $x\nmid y$.
	Clearly, we have that $v_{i}(x)r\le v_{i}(u),$ for every $v_i \in V$, hence
	$$r\le \inf_{v_i\in V}\bigg\lfloor\frac{v_{i}(u)}{v_{i}(x)}\bigg\rfloor.$$ Let $s:=\inf_{v_i\in V}\Big\lfloor\frac{v_{i}(u)}{v_{i}(x)}\Big\rfloor.$ Then $u/x^s\in\UU$ since $$v_{i}(ux^{-s})=v_{i}(u)-sv_{i}(x)\ge v_{i}(u)-\bigg\lfloor\frac{v_{i}(u)}{v_{i}(x)}\bigg\rfloor v_{i}(x)\ge  v_{i}(u)-\frac{v_{i}(u)}{v_{i}(x)}v_{i}(x)=0,$$ hence $$(x\mid u)_v=\inf_{v_i\in V}\frac{v_{i}(u)}{v_{i}(x)}.$$

    \begin{remark}
    If $\UU$ is a Krull domain, then the discrete valuations $V$ are just those arising from the finitely many height-1 prime ideals of $\UU$ containing $x$, so in this case
    \[
    (x\mid u)_v=\min_{\substack{\mathfrak p \in \mathfrak X(\UU)\\ x \in \mathfrak p}} \bigg\lfloor \frac{v_{\mathfrak{p}}(u)}{v_{\mathfrak{p}}(x)} \bigg\rfloor.
    \]
    \end{remark}	 
	
	\bibliographystyle{alpha}
	\bibliography{cluster.bib}

\end{document}